\documentclass{svjour3}
\smartqed
\usepackage{amsmath,amsfonts}
\usepackage{color}
\usepackage{graphicx}
\usepackage{epsfig}
\newcommand{\A}{\mathcal{A}}
\newcommand{\C}{\mathcal{C}}
\newcommand{\G}{{\mathcal G}}
\newcommand{\p}{{\mathcal P}}
\newcommand{\pw}{\varPi_{{k}}}
\newcommand{\iprod}[1]{\langle#1\rangle}

\newcommand{\bigiprod}[1]{\bigl\langle#1\bigr\rangle}
\newcommand{\W}{{\mathcal W}}
\newcommand{\I}{{\mathcal I}}
\newcommand{\R}{\mathbb{R}}
\newcommand{\cD}{\rm ^c D}
\begin{document}
\title{A discontinuous Galerkin  method for
time fractional diffusion equations with variable coefficients\thanks{The support of  the Science Technology Unit at KFUPM through  King Abdulaziz City for Science and
Technology (KACST) under National Science, Technology and Innovation Plan (NSTIP) project No. 13-MAT1847-04
 is gratefully acknowledged.}}

\date{\today}
\journalname{Numerical Algorithms}

\author{K. Mustapha$^1$, B. Abdallah$^2$, K.M. Furati$^3$ \and M. Nour$^4$}
\institute{
Kassem Mustapha, \email{kassem@kfupm.edu.sa}\\
Khaled  Furati, \email{kmfurati@kfupm.edu.sa}
\at  Department of Mathematics and Statistics,
King Fahd University of Petroleum and Minerals, Dhahran, 31261, Saudi Arabia.\\
}
\date{\today}

\maketitle

\begin{abstract}
 We propose a piecewise-linear, time-stepping discontinuous Galerkin  method  to solve numerically a time fractional diffusion equation involving Caputo derivative of order $\mu\in (0,1)$ with variable coefficients. For the spatial discretization, we apply the standard piecewise linear continuous Galerkin method. Well-posedness of the fully discrete  scheme and error analysis will be shown.
For a time interval~$(0,T)$ and a spatial domain~$\Omega$,  our analysis suggest that the error in
$L^2\bigr((0,T),L^2(\Omega)\bigr)$-norm is of order $O(k^{2-\frac{\mu}{2}}+h^2)$  (that is, short by order $\frac{\mu}{2}$
 from being optimal in time) where $k$ denotes the maximum time step, and $h$ is the maximum diameter of the
elements of the (quasi-uniform) spatial mesh. However, our numerical experiments indicate  optimal  $O(k^{2}+h^2)$ error bound in the stronger $L^\infty\bigr((0,T),L^2(\Omega)\bigr)$-norm. Variable time steps are used to compensate the  singularity of the continuous solution  near ~$t=0$.
\end{abstract}

\section{Introduction}
In this paper,   we investigate a numerical solution that allows a time discontinuity  for  solving time fractional diffusion equations with variable diffusivity. Let $\Omega$ be a bounded convex polygonal domain in~$\R^d$ ($d=1,2,3$),
 with a boundary $\partial \Omega$, and $T > 0$ be a fixed time.
Then the fractional model problem  is given by:
\begin{equation}\label{eq: ivp}
\begin{aligned}
 {\cD}^{\mu} u(x,t)-\nabla \cdot (\A(x,t) \nabla u(x,t)) &= f(x,t) \quad &&{\rm on
}~~\Omega\times (0,T],\\
 u(x,0)&= u_0(x) &&{\rm  on }~~ \Omega,\\
 u(x,t)&= 0  &&{ \rm on
} ~~\partial \Omega\times (0,T],
\end{aligned}
\end{equation}
 where  we assume that $\A \in \C^1([0,T],L^\infty(\Omega))$ and  satisfies
\begin{equation}\label{eq: A positive}
0<a_{\min} < \A (x,t) < a_{\max}<\infty\quad { \rm on
} ~~\overline \Omega\times [0,T].\end{equation}
 Here,  ${\cD}^{\mu}$ is the Caputo's fractional derivative  defined by
\[
{\cD}^{\mu} v(t)=I^{1-\mu}
v'(t):=\int_0^t\omega_{1-\mu}(t-s)\,v'(s)\,ds\quad\text{with}\quad
\omega_{1-\mu}(t):=\frac{t^{-\mu}}{\Gamma(1-\mu)},
\]
where throughout the paper, $0<\mu <1$. Noting that,  $I^{1-\mu}$ is the Riemman Liouville fractional integral operator, and $v'$ denotes the time partial derivative of $v$.

Over the past few decades, researchers have observed numerous biological, physical and financial systems in which some key underlying random motion  conform to a model where the  diffusion is anomalously slow (subdiffusion) and not to the classical model of diffusion. For instance, the fractional diffusion model problem \eqref{eq: ivp} is known to capture well the dynamics of subdiffusion processes, in which the mean square variance grows at a rate slower than that in a Gaussian process, see \cite{Podlubny1999}. Fractional diffusion has been successfully used to describe diffusion in media with fractal geometry \cite{Nigmatulin1986}, highly heterogeneous aquifer \cite{AdamsGelhar1992},
  and underground environmental problem \cite{HatanoHatano1998}.  Two distinct approaches can be used for modelling fractional sub-diffusion. One based on fractional Brownian motion  and Langevin equations \cite{MandelbrotNess1968,Wang1992}, this leads to a diffusion equation with a varying diffusion coefficient exhibiting a fractional power law scaling in time \cite{Wang1992}. The other based on continuous time random walks  and master equations with power law waiting time densities  which leads to a diffusion equation with fractional order temporal derivatives operating on the spatial Laplacian \cite{MetzlerKlafter2000}.

The innovation of this paper is to investigate  the piecewise linear time-stepping discontinuous Galerkin (DG) method   combined with the standard finite elements (FEs) in space  for solving numerically  time fractional models with
  variable diffusion coefficients of the form \eqref{eq: ivp}. Since their inception in the early 1970s, DG methods have found numerous applications \cite{CockburnKarniadakisShu2000}, including for the time discretization of fractional diffusion and fractional wave equations, \cite{MustaphaMcLean2009,MustaphaMcLean2011,MustaphaMcLean2013}.   Their advantages include excellent stability properties  and suitability for adaptive refinement based on a posteriori error estimates \cite{ErikssonJohnson1991} to handle problems with low regularity. The present work is motivated  by an earlier paper \cite{MustaphaMcLean2011}. There in, the first author and McLean considered a piecewise-linear DG method for a fractional diffusion problem with a constant diffusivity:
\begin{equation}
\label{eq: reimann} u'(x,t) - {^{R}{\rm D}}^{1-\mu} \Delta u(x,t)  = f(x,t) \quad\mbox{ for } (x,t)\in \Omega\times (0,T],
\end{equation}
where  $^{R}{\rm D}^{1-\mu} u:=\frac{\partial }{\partial t}( I^\mu u)$ (Riemann--Liouville fractional derivative). Recently, high order $hp$-DG methods with exponential rates of convergence for fractional diffusion \eqref{eq: reimann} and also for fractional wave equations were studied in \cite{Mustapha2015,MustaphaSchoetzau2013}. Noting that, when $\A$ is constant and $f\equiv 0$ in \eqref{eq: ivp}, one may look at  \eqref{eq: reimann} as an alternative representation of \eqref{eq: ivp}.

Numerical solutions for model problems of the form \eqref{eq: ivp} with  constant diffusion parameter  $\A$  have attracted considerable interest in recent years. The case of variable coefficients is indeed very
interesting and also practically important. However, due to the additional difficulty in this case, there are only few papers in the existing literature.   With $\Omega=(0,L)$, Alikhanov \cite{Alikhanov2015} constructed a new difference analog of the time fractional Caputo derivative with the order of approximation $O(k^{3-\mu})$. Difference schemes of order $O(h^q+k^2)$ (with $q\in\{2,4\}$) were proposed and analyzed assuming that $u$ is sufficiently regular, where $k$ is the temporal grid size and $h$ is the spatial grid size. For a time independent diffusivity,  Zhao and Xu \cite{ZhaoXu2014} proposed a compact difference scheme for \eqref{eq: ivp}. Stability and convergence properties of the scheme were  proved. For time fractional convection-diffusion problems, Cui \cite{Cui2015} studied a compact exponential scheme. The stability and the convergence analysis were showed assuming that the coefficients of the model problem are constants. For time independent  coefficients, Saadatmandi et al.  \cite{SaadatmandiDehghanAzizi} investigated the Sinc-Legendre collocation method.

For {\em one-dimensional} spatial domains and constant diffusion parameter  $\A$, Murio \cite{Murio2008} and Zhang et al. \cite{ZhangLiao2014}  studied two classes of  finite difference (FD) methods. Stability properties were  provided. Another  FD scheme in time (with $L1$ approximation for the Caputo fractional derivative) combined with the spatial fourth order compact difference approach was studied  by Ren et al.  \cite{RenZhao2013}. Convergence rates of order $O(k^{1+\mu}+ h^4)$ were proved. Murillo and Yuste \cite{MurilloYuste2013} presented an implicit FD method over non-uniform time steps. An adaptive procedure was described to choose the size of the time meshes.  Lin and Xu \cite{LinXu2007} combined  a  FD scheme  in time and a   spectral method in space. Accuracy of order $ O(k^{1+\mu}+ r^{-m})$ was proved, where $r$ is the spatial polynomial degree, and $m$ is related to the  regularity of the exact solution $u$. Later, Li and Xu  \cite{LiXu2009} developed  and analyzed a time-space spectral method. Zhao and Sun \cite{ZhaoSun2011} combined an order reduction approach and $L1$ discretization of the fractional derivative.
 A box-type scheme was constructed and a convergence rate of  order  $ O(k^{1+\mu} + h^{2})$ had been proved. Finite central differences in time combined with  the FE method in space was studied by Li and Xu \cite{LiXu2010}. For a smooth $u$, a convergence rate of order $O(k^2+h^{\ell+1})$ was achieved where $\ell$ is the degree of the FE solutions in space. Recently, a similar convergence rate was shown by Zeng et al. \cite{ZengTurner2015-2} where the fractional linear multistep method was used for the time discretization. For a high-order local DG  method for space discretization, we refer to the work by  Xu and Zheng \cite{XuZheng2013}.

 For {\em two-} or {\em three-dimensional} spatial domains with $\A=1$ in \eqref{eq: ivp}, Brunner et al. \cite{BrunnerYamamoto2010} used  an algorithm that couples an adaptive time
stepping and adaptive spatial basis selection approach for the numerical solution of \eqref{eq: ivp}. A semi-discrete  piecewise linear Galerkin FE   and lumped mass Galerkin methods were studied  by Jin et al.
\cite{JinLarzarov2013}. An optimal error with respect to the
regularity error estimates was established for $f\equiv 0$ and non-smooth initial data $u_0$. Cui \cite{Cui2013} studied the convergence analysis of compact alternating direction implicit (ADI) schemes for sufficiently smooth solutions of \eqref{eq: ivp}. For {\em three-dimensional} spatial domains,
a fractional ADI scheme was  proposed and analyzed by Chen et al. \cite{ChenLiuLiuChenAnhTurnerBurrage2014}.
 Mustapha et al. \cite{MustaphaAbdallahFurati2014}  proposed {\em low-high} order time stepping discontinuous
Petrov-Galerkin  methods combined with
 FEs  in space.  Using variable time meshes,  $O(k^{m+(1-\mu)/2}+h^{r+1})$ convergence rates were shown, where   $m$ and $r$ are  the degrees of approximate solutions in the  time and spatial variables, respectively. Optimal convergence rates in both variables were demonstrated numerically. In \cite{MustaphaNourCockburn2015}, a hybridizable DG method in space  was extensively studied by Mustapha et al..

The outline of the paper is as follows.  Section~\ref{sec: numerical method} introduces a fully discrete DG FE  scheme. In
Section~\ref{sec: stability}, we prove the stability of the discrete solution and provide a remark about the existence and uniqueness of the numerical solution. Section \ref{sec: projection} is devoted to introduce time and space projection operators that will be used later to show the convergence of the numerical scheme. The error analysis is given in
Section~\ref{sec:error}.   Using suitable refined time-steps (towards $t=0$) and quasi-uniform spatial meshes,  in the $L^2((0,T),L^2(\Omega))$-norm,
convergence of order $O(k^{2-\frac{\mu}{2}}+h^2)$ is achieved. Section~\ref{sec: Numerical} is dedicated to present a sample
of numerical test which illustrate that our error bounds are
pessimistic. For a strongly graded time mesh, in the stronger $L^\infty((0,T),L^2(\Omega))$-norm, we observe optimal convergence rates,  that is, error of order $O(k^{2}+h^2)$.
\section{The numerical method}\label{sec: numerical method}
To describe our fully discrete DG FE method, we introduce a  time partition of the interval $[0,T]$ given by the points:
$0=t_{0}<t_{1}<\cdots<t_{N}=T\,.$ We set $I_n=(t_{n-1},t_n)$ and $k_n=t_n-t_{n-1}$ for $1\le n\le N$ with $k:=\max_{1\le n\le
N}k_n$. Let $S_h\subseteq H_0^1(\Omega)$ denotes the space of continuous, piecewise polynomials
 of total degree $\le 1$  with respect to a quasi-uniform partition of~$\Omega$ into
conforming triangular finite elements, with maximum diameter~$h$.
Next, we introduce our time-space finite dimensional DG FE  space:
\[
 \W=\{w \in L^2((0,T), S_h):~~w|_{I_{n}}\in
\p_1(S_h)~{\rm for}~1\le n\le N\} \]  where $\p_1(S_h)$ denotes the space
of linear polynomials in the time variable $t$,  with coefficients in~$S_h$.
We denote the left-hand limit, right-hand limit and
jump at~$t_n$ by
\[
w^n:=w(t_n)=w(t_n^-),\quad w^n_+:=w(t_n^+),\quad [w]^n:=w^n_+-w^n,
\]
respectively. The weak form of the fractional diffusion equation in~\eqref{eq:
ivp} is
\begin{equation}\label{eq: weak In}
\int_{I_n}\bigl[\iprod{{\cD}^{\mu}  u,v}+a\bigl(t,
u,v\bigr)\bigr]\,dt
    =\int_{I_n}\iprod{f,v}\,dt,\quad \forall ~v\in L^2\bigl(I_n, H^1(\Omega)\bigr)\,.
\end{equation}
 Throughout the paper, $\iprod{\cdot,\cdot}$  denotes the $L^2$-inner product and $\|\cdot\|$ is the associated norm, and $\|\cdot\|_m$ ($m\ge 1$) denotes the norm on the Sobolev space $H^m(\Omega)$.  We use  $\| \cdot\|_{L^q(Y)}$ ($q\ge 1)$ to denote the norm on
$L^q((0,T),Y(\Omega))$ for any Sobolev space $Y(\Omega)$.

For each fixed $t \in (0,T]$, $a(t,\cdot, \cdot):H_0^1 (\Omega)\times H_0^1(\Omega)\rightarrow\mathbb{R}$ 
 is the bilinear form
$$ a(t,v,w)= \iprod{\A(\cdot,t) \nabla v,\nabla w}= \int_{\Omega} \A (x,t)\nabla v(x)\cdot \nabla w(x)\,dx$$
associated with the operator $\nabla \cdot (\A (\cdot,t) \nabla)$ which is symmetric and positive definite (by \eqref{eq: A positive}), that is,
 there exist  positive  constants $c_0$ and $c_1$ such that
\begin{equation}\label{pd}
c_0\| v(t)\|_1^2 \leq |v(t)|_1^2:=a(t,v,v) \leq c_1 \|v(t)\|_1^2\quad \forall ~~v(t) \in H_0^1(\Omega)\,.
\end{equation}
The DG FE approximation $U\in \W$ is
defined as follows: Given $U(t)$ for $0\le t\le t_{n-1}$, the
solution $U\in \p_1(S_h)$ on $I_n$
is determined by requesting that for $1\le n\le N$,
\begin{multline*}
\int_{I_n}\bigl[\iprod{{\cD}_{dg}^{\mu}  U+
    \sum_{j=0}^{n-1}\omega_{1-\mu}(t-t_j)\,[U]^j,X}
    +a\bigl(t, U,X\bigr)\bigr]\,dt
    =\int_{I_n}\iprod{f,X}\,dt,~ \forall~X\in \p_1(S_h),
\end{multline*}
 with $U^0_+=U^0 \in S_h$ is a suitable approximation of the initial  data $u_0$, where
\[ {\cD}^{\mu}_{dg}  U(t):=\sum_{j=1}^n\int_{t_{j-1}}^{\min\{t_j,t\}}\omega_{1-\mu}(t-s)\,U'(s)\,ds\quad {\rm for}~~t\in I_n\,.\]
  Since
 \begin{equation}\label{eq: Dmu U}
 \begin{aligned}
{\rm ^R D}^\mu U(t)&:=  \frac{\partial}{\partial
t}\int_0^t\omega_{1-\mu}(t-s)U(s)\,ds\\
&={\cD}^{\mu}_{dg}U(t)+\omega_{1-\mu}(t)U^0
+\sum_{j=1}^{n-1}\omega_{1-\mu}(t-t_j)\,[U]^j\quad {\rm for}~~t\in I_n,
\end{aligned}
\end{equation}
our scheme can be rewritten in a compact form as follows: for $1\le n\le N$,
\begin{multline}\label{eq: DGFM}
\int_{I_n}\bigl[\bigiprod{{\rm ^R D}^\mu  U,X}
    +a\bigl(t, U,X\bigr)\bigr]\,dt
    =\int_{I_n}\iprod{f+\omega_{1-\mu}(t)U^0,X}\,dt\quad \forall~X\in \p_1(S_h).
\end{multline}
Noting that, since the DG FE scheme \eqref{eq: DGFM} amounts to a square linear system, the existence of the numerical solution $U$ follows from its uniqueness. The uniqueness follows immediately from the above stability property in Theorem \ref{th: stability}.

\section{Stability of the  numerical solution}\label{sec: stability}
 To show the stability of the DG FE scheme \eqref{eq: DGFM}, we claim first the identity: $v(t)=I^{\mu}({\rm ^R D}^\mu v)(t)$ for any $v \in \W.$ If $v$ is an absolutely  continuous function in the time variable, this identity follows by  applying the fractional integral operator $I^\mu$ to both sides of the equality ${\rm ^R D}^\mu v(t)={\cD}^{\mu}v(t)+\omega_{1-\mu}(t)v(0)$   and then   changing the order of integrals and using the identity: $\int_{s}^{\tilde t}
\omega_{1-\mu}(t-s)
\omega_{\mu}(\tilde t -t)\,dt=1.$
\begin{lemma}\label{eq: Ialpha Dalpha identity}If $v\in \W$, then
$$v(t)=I^{\mu}({\rm ^R D}^\mu v)(t)\quad {\rm  for~~}  t \in I_n~~{\rm  with}~~1\le n\le N.$$
\end{lemma}
\begin{proof}
Since $v$ has possible discontinuities at the time nodes $t_0,t_1,\cdots, t_{j-1}$, from \eqref{eq: Dmu U},
\begin{equation}\label{eq:formula for Dalpha}
{\rm ^R D}^\mu v(s)=\omega_{1-\mu}(s)v^0_+
+\sum_{i=1}^{j-1}\omega_{1-\mu}(s-t_i)\,[v]^i+{\cD}^{\mu}_{dg} v(s) ~~{\rm for}~~s \in I_j\,.
\end{equation}
Applying the operator $I^{\mu}$ to both sides and using      $I^{\mu}({\cD}^{\mu} v)(t)=\int_0^t v'(s)\,ds$, we observe
\begin{multline*}
I^{\mu}({\rm ^R D}^\mu v)(t)=v^0_++\sum_{j=2}^{n-1}\int_{I_j}
\omega_{\mu}(t-s)
\sum_{i=1}^{j-1}\omega_{1-\mu}(s-t_i)\,[v]^i\,ds\\+\int_{t_{n-1}}^{t}
\omega_{\mu}(t-s)
\sum_{i=1}^{n-1}\omega_{1-\mu}(s-t_i)\,[v]^i\,ds+\int_0^t v'(s)ds~~{\rm for}~~t \in I_n\,.
\end{multline*}
Now, changing the order of summations and rearranging the
terms yield
\begin{multline*}
I^{\mu}({\rm ^R D}^\mu v)(t)=v^0_++\sum_{i=1}^{n-2}\sum_{j=i+1}^{n-1}\int_{I_j}
\omega_{\mu}(t-s)
\omega_{1-\mu}(s-t_i)\,[v]^i\,ds\\+\sum_{i=1}^{n-1}\int_{t_{n-1}}^{t}
\omega_{\mu}(t-s)
\omega_{1-\mu}(s-t_i)\,[v]^i\,ds+\sum_{j=1}^{n}\int_{t_{j-1}}^{\min
\{t,t_j\}} v'(s)ds \\=v^0_++\sum_{i=1}^{n-2}\int_{t_{i}}^{t}
\omega_{\mu}(t-s)
\omega_{1-\mu}(s-t_i)\,[v]^i\,ds\\+\int_{t_{n-1}}^{t}
\omega_{\mu}(t-s)
\omega_{1-\mu}(s-t_{n-1})\,[v]^{n-1}\,ds+\sum_{j=1}^{n}\int_{t_{j-1}}^{\min
\{t,t_j\}} v'(s)ds\,.
\end{multline*}
Integrating and simplifying, then we have
\begin{align*}
I^{\mu}({\rm ^R D}^\mu v)(t)&=v^0_++\sum_{i=1}^{n-1}\,[v]^i+\sum_{j=1}^{n-1}\,(v^j-v^{j-1}_+)+v(t)-v^{n-1}_+=v(t)~~{\rm for}~~t \in I_n\,.\quad \Box
\end{align*}
\end{proof}

In the next lemma we state some important properties of the Riemman Liouville factional  operators. These properties will be used to show the stability of the numerical scheme, as well as, in our  error analysis in the forthcoming section.

For $\ell \in \{0,1\}$, we let $ \C^\ell(J_n,L^2(\Omega))$ ($J_n:=\cup_{j=1}^n I_j$) denote the space of functions $v: J_n \rightarrow L^2(\Omega)$ such that
the restriction $v|_{I_j}$ extends to an $\ell$-times continuously differentiable function on the
closed interval $\overline I_j$ for $1 \le j \le  n.$ For later use,  we let
\[\|v\|_{I_j}:=\sup_{t\in I_j}\|v(t)\|\quad{\rm and}\quad \|v\|_{J_n}:=\max_{j=1}^n \|v\|_{I_j},\]
where we drop $n$ when $J_n=J_N$.
\begin{lemma}
\label{lemma:positivity-equal-0}
For  $1 \le n \le N$  and for  $0<\alpha<1,$ we have

 {\rm (i)} {The~  operator $^{R}{\rm D}^\alpha$ satisfies: for $v \in  \C^1(J_n,L^2(\Omega))$},
$$\int_0^{t_n} \iprod{^{R}{\rm D}^\alpha v,v}\,dt \ge \frac{(\pi \alpha)^\alpha}{(1+\alpha)^{1+\alpha}}\cos\Big( \frac{\alpha\pi}{2}\Big) t_n^{-\alpha} \int_0^{t_n}\|v(t)\|^2\,dt\,.$$

{\rm (ii)} {The integral operator  $I^{\alpha}$ satisfies: for $\, v,\,w \in \C^0(J_n,L^2(\Omega))$}
$$ \left|\int_0^{t_n}\iprod{I^\alpha v,w}\,dt\right|^2
    \le \sec^2\Big( \frac{\alpha\pi}{2}\Big)\,\int_0^{t_n}\iprod{I^\alpha v,v}\,dt \int_0^{t_n}\iprod{I^\alpha w,w}\,dt\,.
$$
\end{lemma}
\begin{proof}
The property (i)  was proven in \cite[Theorem A.1]{McLean2012} by using the Laplace transform and Plancherel Theorem.
 For the proof of the property (ii),  see   \cite[Lemma 3.1]{MustaphaSchoetzau2013}. $\quad \Box$
\end{proof}

The next theorem shows the stability of the  DG FE scheme.
\begin{theorem}\label{th: stability} Assume that $U^0 \in L^2(\Omega)$ and $f \in L^2((0,T),L^2(\Omega))$. Then,
\[
\int_0^T\|U\|_1^2\,dt
    \leq C T^{1-\mu}\|U^0\|^2
     +C\,\int_0^{T}\|f\|^2\,dt\,.
\]
\end{theorem}
\begin{proof} Choosing $ X=U$ in the DG FE scheme \eqref{eq: DGFM},
 and then summing over $n$, we obtain
 \[
\int_0^{T}\bigl[\bigiprod{{\rm ^R D}^\mu  U,U}
    +a(t,U,U)\bigr]\,dt
    =\int_0^{T} \iprod{ f+\omega_{1-\mu}(t) U^0,U}\,dt.
\]
Since $a(\cdot,U,U)\ge c_0\|U\|_1^2$ by \eqref{pd} and $\iprod{f,U}\le \frac{1}{2c_0}\|f\|^2+\frac{c_0}{2}\|U\|^2$, we have
\[
\int_0^{T}\bigl[\bigiprod{{\rm ^R D}^\mu  U,U}
    +\frac{c_0}{2}\|U\|_1^2\bigr]\,dt
    \le \int_0^{T} \Big(\iprod{\omega_{1-\mu}(t) U^0,U}+\frac{1}{2c_0}\|f\|^2\Big)\,dt\,.
\]
Using the identity   $U(t)=I^{\mu}({\rm ^R D}^\mu U)(t)$ from Lemma  \ref{eq: Ialpha Dalpha identity},  Lemma
\ref{lemma:positivity-equal-0} (ii),  the inequality $ab\le \frac{a^2}{4}+b^2$, and the identity $I^{\mu} \omega_{1-\mu}(t)=1$,  yield
\begin{multline}\label{eq: U0 RL}
 \int_0^{T} \iprod{\omega_{1-\mu}(t) U^0,U}\,dt=\int_0^{T} \iprod{ \omega_{1-\mu}(t) U^0,I^{\mu}({\rm ^R D}^\mu U)}\,dt
    \\ \leq \frac{1}{4}\int_0^{T}\bigiprod{{\rm ^R D}^\mu  U,U}\,dt+\sec^2(\mu
\pi/2)
   \int_0^{T}\omega_{1-\mu}(t)(I^{\mu} \omega_{1-\mu})(t)\,dt\|U^0\|^2 \\ \leq \frac{1}{4}\int_0^{T}\bigiprod{{\rm ^R D}^\mu  U,U}\,dt+C\,T^{1-\mu}\|U^0\|^2\,.
\end{multline}
 To complete the proof, we combine the above two equations and use the positivity property of the operator ${\rm ^R D}^\mu$ given by Lemma
\ref{lemma:positivity-equal-0} (i).$\quad \Box$
\end{proof}
\section{Projections and errors}\label{sec: projection}
In this section, we introduce time and space projections, and then  derive some bounds
and errors properties that will be used later in our convergence analysis.
\subsection{Projection in space}
For each $t\in [0,T]$, the elliptic  projection operator $\,R_h:H^1_0(\Omega) \to S_h$ is defined by
 \begin{equation}\label{eq: Ritz projection} a(t,R_h v-v, \chi)=0\quad \forall~~ \chi\in S_h\,.\end{equation}
By the assumption  $\A \in \C^1([0,T],L^\infty(\Omega)),$ for each $t\in (0,T)$, the projection error $\xi:=R_h u -u$ has the well-known approximation property:
\begin{equation}\label{eq:estimate ritz projector 1} \|  \xi(t)\|+h\|\nabla \xi(t)\|\leq C\,
h^2\|u(t)\|_2\quad {\rm for}~~u(t)\in H^2(\Omega) \cap  H^1_0(\Omega)\,.
\end{equation}
Moreover, we have
\begin{equation}\label{eq:estimate ritz projector 2}  \|  \xi'(t)\|\leq C\,h^2(\|u(t)\|_2+\|u'(t)\|_2)\quad {\rm for}~~u(t),\,u'(t) \in H^2(\Omega) \cap  H^1_0(\Omega)\,.
\end{equation}
 To show \eqref{eq:estimate ritz projector 2}, we decompose
$\xi'$ as: $\xi'=g_h +(R_h u'-u')$  where for each $t\in (0,T)$, $g_h(t)=(R_h u)'(t)-R_h u'(t)\in S_h$. Since $\|(R_h u'-u')(t)\|\le C\,h^2\|u'(t)\|_2$ by the approximation property \eqref{eq:estimate ritz projector 1} applied to $u'$,
 it remains to derive a similar bound for $g_h$.  To do so, we use the Nitsche’s trick: for each $t\in (0,T)$, there is $\phi \in H^2(\Omega)$  such that
\[-\nabla \cdot (\A(t) \nabla \phi)=g_h(t)\quad{\rm in}~~\Omega,\quad \phi=0~~{\rm on}~~\partial \Omega\]
with $\|\phi\|_2\le C\|\nabla \cdot (\A(t) \nabla \phi)\|$ which holds for convex polyhedral domains. Taking the inner product with
 $g_h(t) \in S_h$, and then using the orthogonality
property of $R_h$,
\[\|g_h(t)\|^2=a(t,\phi, g_h(t))=a(t,R_h\phi, g_h(t))=a(t,g_h(t),R_h\phi)\,.\]
But, by  the definition of $R_h$, for each $t\in (0,T)$,
\begin{equation}\label{eq: g_h}
\begin{aligned}
a(t, g_h(t), \chi)&=a(t, (R_h u)' (t), \chi)-a(t, u' (t), \chi)\\
&=\frac{d}{dt}a(t, \xi (t), \chi)-\iprod{\A'(t) \nabla \xi(t),\nabla \chi}\\
&=-\iprod{\A'(t)\nabla \xi(t),\nabla \chi}\quad \forall~~ \chi\in S_h\,.\end{aligned}
\end{equation}
Therefore, \begin{multline*}
\|g_h(t)\|^2=
-\iprod{\A'(t)\nabla \xi(t),\nabla R_h \phi}
=\iprod{\A'(t)\nabla \xi(t),\nabla (\phi-R_h \phi)}+\iprod{ \xi(t),\nabla\cdot(\A'(t)\nabla \phi)}\,.\end{multline*}
Finally, using the Cauchy-Schwarz inequality, the approximation property in \eqref{eq:estimate ritz projector 1}, and the inequality $\|\phi\|_2\le C\|\nabla \cdot (\A(t) \nabla \phi)\|$,  we observe
\[\|g_h(t)\|^2\le C\|\nabla (\phi-R_h \phi)\|\,\|\nabla \xi(t)\|+C\|\xi(t)\|\,\|\phi\|_2\le C\,h^2\|g_h(t)\|\,\|u(t)\|_2\,.\]
The proof of \eqref{eq:estimate ritz projector 2} is completed now.

\subsection{Projection in time}
The local $L^2$-projection operator $\pw : L^2( I_{n},L^2(\Omega)) \to
\mathcal{C}( I_{n},\p_1(L^2(\Omega))$  defined by:
 \[\int_{I_n} \iprod{\pw v-v,\,w}\,dt=0~~\forall~~w \in \p_1(L^2(\Omega))\quad{\rm for}~~1\le n\le N,\]
where  $\p_{1}(L^2(\Omega))$ is the space
of linear polynomials in the time variable $t$,  with coefficients in~$L^2(\Omega)$\,. Explicitly,
\[
\pw v(t)=\frac{12}{k_n^3}(t-t_{n-\frac{1}{2}})\int_{I_n}(s-t_{n-\frac{1}{2}})\,v(s)\,ds+\frac{1}{k_n}\int_{I_n}v(s)\,ds
    \quad\text{for $t\in I_n$\,},
\]
where $t_{n-\frac{1}{2}}:=(t_{n-1}+t_n)/2.$ Hence, for $v'\in L^1( I_{n},L^2(\Omega))$,
\[
(\pw v)'(t)=\frac{12}{k_n^3}\int_{I_n}(s-t_{n-\frac{1}{2}})\,v(s)\,ds=\frac{6}{k_n^3}\int_{I_n}(t_n-t)(s-t_{n-1})\,v'(s)\,ds\,.
\]
Thus,   for $1\le n\le N$, we have
\begin{equation}\label{(ii)}
\|\pw v(t)\|  \leq \frac{4}{k_n}\int_{I_n}\|v(s)\|\,ds\quad{\rm and}\quad  \|(\pw v)'(t)\| \leq \frac{3}{2k_n}\int_{I_n}\|v'(s)\|\,ds\,.
\end{equation}
Setting $\eta_v=\pw v-v$, we have the well-known projection error bound: for $t\in I_n$,
\begin{equation}\label{projection error}
\|\eta_v(t)\|+k_n\,\|\eta_v'(t)\|\le C\, k_n^{\ell-1}\int_{I_n}\Big\|\frac{\partial^\ell v}{\partial t^\ell} (s)\Big\|\,ds\quad{\rm for}~~\ell=1,\,2.
\end{equation}
Next, we show an error bound property of $\pw$ that involves the  operator ${\rm ^R D}^\mu.$
\begin{lemma}\label{lem: estimat of Dalpha eta} Let  $\frac{\partial^\ell v}{\partial t^\ell} \in L^1((0,t_n),L^2(\Omega))$ for $\ell\in \{1,\,2\}$.  We have
\begin{align*}
\int_{I_n} \bigiprod{{\rm ^R D}^\mu  \eta_v,\eta_v}\,dt &\leq
   C\,k_n^{1-\mu} \max_{j=1}^{n} k_j^{2\ell-2}\Big(\int_{I_j}\Big\|\frac{\partial^\ell v}{\partial t^\ell} \Big\|\, dt\Big)^2\quad{\rm for}~~1\le n\le N.
\end{align*}
 \end{lemma}
\begin{proof} We integrate by parts and  notice that
\begin{equation}\begin{aligned}
\int_{I_n}\bigiprod{{\rm ^R D}^\mu  \eta_v,\eta_v}\,dt
&= \iprod{I^{{1-\mu}}\eta_v(t), \eta_v(t)}\bigg|_{t_{n-1}^+}^{t_n^-} -  \int_{I_n}\iprod{I^{{1-\mu}}\eta_v, \eta_v'}\,dt\\
&=   \iprod{\I^n(t_n), \eta_v(t_n)}- \int_{I_n}\iprod{\I^n(t), \eta_v'(t)}\,dt,
\end{aligned}
\end{equation}
where for $t\in I_n$,
 \begin{align*}\I^n(t)&:=I^{{1-\mu}}\eta_v(t)-I^{{1-\mu}}\eta_v(t_{n-1})\\
 &=\int_0^{t_{n-1}}
    [\omega_{{1-\mu}}(t-s)-\omega_{{1-\mu}}(t_{n-1}-s)] \eta_v(s)\,ds+  \int_{t_{n-1}}^t
    \omega_{{1-\mu}}(t-s)\eta_v(s)\,ds\,.\end{align*}
Simplifying then integrating,    we observe
 \begin{align*} \|\I^n(t)\|
 &\le
   \Big(\int_0^{t_{n-1}}
    [\omega_{{1-\mu}}(t_{n-1}-s)-\omega_{{1-\mu}}(t-s)] \,ds+  \int_{t_{n-1}}^t
    \omega_{{1-\mu}}(t-s)\,ds\Big) \|\eta_v\|_{J_n}\\
    &\le    2\,\omega_{2-\mu}(k_n) \, \| \eta_v\|_{J_n}\quad{\rm for}~~t\in I_n\,.\end{align*}
Therefore, an application of the Cauchy-Schwarz inequality gives
 \begin{align*}
\int_{I_n}|\bigiprod{{\rm ^R D}^\mu  \eta_v,\eta_v}|\,dt    &\le    2\,\omega_{2-\mu}(k_n)\,\| \eta_v\|_{J_n}\left(\| \eta_v(t_n)\|+\int_{I_n}\| \eta_v'\|\,dt\right),
\end{align*}
and hence,  using the error projection  in \eqref{projection error}, we obtain the  desired bound.
$\quad \Box$
\end{proof}

Since $^R{\rm D}^{\mu} v(t)=\omega_{1-\mu}(t)v(0)+I^{{1-\mu}} v'(t)$ and since $\|\eta_v\|_{I_n}\le 5\|v\|_{I_n}$ by the triangle inequality and the first inequality in \eqref{(ii)}, we have
\[\int_{I_n}
|\bigiprod{^R{\rm D}^{\mu} v,\eta_v}|\,dt\le
5\|v\|_{I_n}\int_{I_n}
(\omega_{1-\mu}(t)\|v(0)\|+\|I^{{1-\mu}} v'\|)\,dt.\]
Summing over $n$ gives
\[
\int_0^T
|\bigiprod{^R{\rm D}^{\mu} v,\eta_v}|\,dt\le
5\,\|v\|_J(\omega_{2-\mu}(T)\|v(0)\|+\int_0^T\omega_{2-\mu}(T-s)\| v'(s)\|\,ds)\,.\]
On the other hand, noting that
\[\int_0^T
\bigiprod{^R{\rm D}^{\mu} \pw v,v}\,dt=
\iprod{I^{{1-\mu}} \pw v(T),v(T)}-\int_0^T
\iprod{I^{{1-\mu}} \pw v,v'}\,dt,\]
and hence, by the Cauchy-Schwarz inequality and the first inequality in \eqref{(ii)},
\[
\begin{aligned}\Big|\int_0^T
\bigiprod{^R{\rm D}^{\mu} \pw v,v}\,dt\Big|&\le
\|\pw v\|_{J}\int_0^T\Big[\omega_{1-\mu}(T-t)\|v(T)\|+
\int_0^t\omega_{1-\mu}(t-s)\,ds\|v'(t)\|\Big]dt\\
&\le
4\|v\|_{J}\Big(\|v(T)\|\omega_{2-\mu}(T)+\int_0^{T}\omega_{2-\mu}(t)\|v'(t)\|\,dt\Big)\,.\end{aligned}
\]
We combine the above two inequalities and use that $\|v\|_J\le \|v(0)\|+\int_0^T\|v'\|\,dt$, we obtain the bound below that will be used to show the convergence of our scheme,
 \begin{equation}\label{eq: bound 2}
   \Big|\int_0^T
\bigiprod{^R{\rm D}^{\mu} v,\eta_v}\,dt\Big|+
\Big|\int_0^T
\bigiprod{^R{\rm D}^{\mu} \pw v,v}\,dt\Big| \le C\,T^{1-\mu}\,\Big(\|v(T)\|+\int_0^T\|v'\|\,dt\Big)^2\,.\end{equation}
\section{Error estimates}\label{sec:error}
This section is devoted to investigate the convergence of the DG FE scheme, \eqref{eq: DGFM}. We  decompose the error  as follows:
\begin{equation}\label{eq: U-u=theta+eta}
U-u=\zeta+\pw \xi+\eta_u\quad {\rm with}~~\zeta=U-\pw R_hu\,.\end{equation}
Recall that $\xi=R_h u-u$ and $\eta_u=\pw u-u$. The main task now is
to estimate~$\zeta$. 
\begin{theorem}\label{thm: estimate of theta} Choose $U^0=R_h u_0$. For $1\le n\le N,$ we have
\[
\begin{aligned}
   \|\zeta\|_{L^2(H^1)}^2
&\le C\Big(h^4C_1(k,u)+
 C_2(k,u)+h^2k^2\|u\|_{L^2(H^2)}^2\Big),
\end{aligned}\]
where
\begin{equation}\label{eq: C1C2}
\begin{aligned}
C_1(k,u)&=\max_{n=1}^{N}
 \Big(k_n^{-\frac{\mu}{2}}\int_{I_n}\| u'\|_2\, dt\Big)^2+\Big(\|u_0\|_2+\|u'\|_{L^1(H^2)}\Big)^2,\\
 C_2(k,u)&=\max_{n=1}^N \,k_n^{2\ell-2-\mu}\Big(k_n^{-\mu}\Big(\int_{I_n}\Big\|\frac{\partial^\ell u}{\partial t^\ell} \Big\|\,dt\Big)^2+\Big(\int_{I_n}\Big\|\nabla\frac{\partial^\ell u}{\partial t^\ell} \Big\|\,dt\Big)^2\Big)\,.\end{aligned}
 \end{equation}
     \end{theorem}
\begin{proof} We start our proof by taking the inner product of \eqref{eq: ivp} with $\zeta$, using the identity ${\cD}^{\mu} u(t)={^R}{{\rm D}}^{\mu}u(t)-\omega_{1-\mu}(t)u_0$, and then integrating over the time subinterval $I_n$,
\[
\int_{I_n}\bigl[\bigiprod{{\rm ^R D}^\mu  u,\zeta}
    +a\bigl(t, u,\zeta\bigr)\bigr]\,dt
    =\int_{I_n}\iprod{f+\omega_{1-\mu}(t)u_0,\zeta}\,dt\,.\]
The above equation, the DG FE scheme \eqref{eq: DGFM} and the decomposition in \eqref{eq: U-u=theta+eta} imply
\begin{multline}\label{eq: intermdiate equation 1}
\int_0^T\left(\bigiprod{^R{\rm D}^{\mu} \zeta ,\zeta}+|\zeta|_1^2\right)dt=\int_0^T\iprod{\omega_{1-\mu}(t)\,\xi(0),\zeta}\,dt\\
-\int_0^T\left[\bigiprod{^R{\rm D}^{\mu}
(\pw\xi+\eta_u), \zeta}+a(t, \pw\xi+\eta_u , \zeta)\right] dt\,. \end{multline}
Now, using  the continuity property in  Lemma \ref{lemma:positivity-equal-0} $(ii)$,  we notice that
\[
\begin{split}
\Big|\int_0^T\bigiprod{^R{\rm D}^{\mu} \eta_u,\zeta}\,dt\Big| &\le C\int_0^T \bigiprod{^R{\rm D}^{\mu}
\eta_u,\eta_u}\,dt+\frac{1}{4}\int_0^T\bigiprod{^R{\rm D}^{\mu} \zeta,\zeta}\,dt,\\
 \Big| \int_0^T\bigiprod{^R{\rm D}^{\mu} \pw \xi,\zeta}\,dt\Big| &\le
C\int_0^T \bigiprod{^R{\rm D}^{\mu} \pw \xi,\pw \xi}\,dt+\frac{1}{4}\int_0^T\bigiprod{^R{\rm D}^{\mu}
\zeta,\zeta}\,dt\,.\end{split}\]
In addition, following the steps in \eqref{eq: U0 RL}, we observe
\[
 \int_0^{T} \iprod{\omega_{1-\mu}(t)\, \xi(0),\zeta}\,dt
  \leq \frac{1}{4}\int_0^{T}\bigiprod{{\rm ^R D}^\mu  \zeta,\zeta}\,dt+C\,T^{1-\mu}\|\xi(0)\|^2
\]
 Inserting the above three inequalities  in \eqref{eq: intermdiate equation 1}, then simplifying, and using  the positivity property of  ${\rm ^R D}^\mu$, Lemma \ref{lemma:positivity-equal-0} (i), yield
\begin{multline}\label{eq: intermdiate equation 2}
\int_0^T|\zeta|_1^2dt\le C\,T^{1-\mu}\|\xi(0)\|^2
+C\int_0^T\left(\bigiprod{^R{\rm D}^{\mu} \eta_u,\eta_u}+
\bigiprod{^R{\rm D}^{\mu} \pw\xi,\pw\xi}\right)dt\\+\sum_{n=1}^N \Big|\int_{I_n} a(t, \pw\xi+\eta_u , \zeta) dt\Big|\,. \end{multline}
From the definitions of the time projection $\pw$ and the space projection $R_h$,
\begin{multline*}
\int_{I_n}\iprod{\A(t_n)\nabla (\pw\xi+\eta_u), \nabla \zeta}\,dt
=\int_{I_n}\iprod{\A(t_n)\nabla \xi , \nabla \zeta}\,dt
=\int_{I_n}\iprod{[\A(t_n)-\A(t)]\nabla \xi , \nabla \zeta}\,dt
\end{multline*}
and so,
\begin{align*}
\Big|\int_{I_n}a(t, \pw\xi+&\eta_u , \zeta)dt\Big|\\
&=\Big|\int_{I_n}\iprod{\A(t_n)\nabla (\pw\xi+\eta_u) +[\A(t)-\A(t_n)]\nabla (\pw\xi+\eta_u) , \nabla \zeta}\,dt\Big|\\
&=\Big|\int_{I_n}\iprod{[\A(t)-\A(t_n)]\nabla (\eta_\xi+\eta_u) , \nabla \zeta}\,dt\Big|\\
&\le Ck_n \int_{I_n}\|\nabla (\eta_\xi+\eta_u)\|\,\| \nabla \zeta \|\,dt\,. \end{align*}
Thus, by the inequality  $\|\nabla \eta_\xi(t)\|\le \|\nabla \xi(t)\|+4k_n^{-1}\int_{I_n}\|\nabla \xi(s)\|\,ds$ (follows from the triangle inequality and the first property of $\Pi_k$ in \eqref{(ii)})  for $t\in I_n$, and property \eqref{pd},
\[
\Big|\int_{I_n}a(t, \pw\xi+\eta_u , \zeta)dt\Big|
\le C\,k_n^2\int_{I_n}(\|\nabla \xi\|^2+\|\nabla\eta_u\|^2)\,dt+\frac{1}{2}\int_{I_n} |\zeta|_1^2\,dt
\,.\]
Inserting this in \eqref{eq: intermdiate equation 2} and using  \eqref{eq:estimate ritz projector 1} for $t=0$,   we get
\begin{multline*}\int_0^T|\zeta|_1^2dt\le C\,h^4\|u_0\|_2^2 \\
+C\sum_{n=1}^N\int_{I_n}\Big(\bigiprod{^R{\rm D}^{\mu} \eta_u,\eta_u}+
\bigiprod{^R{\rm D}^{\mu} \pw\xi,\pw\xi}+ k_n^2(\|\nabla \xi\|^2+\|\nabla\eta_u\|^2)\Big)\,dt\,.\end{multline*}
But, for $t\in I_n$ and for $\ell \in \{1,\,2\}$,
\begin{align*}
\int_{I_n}\bigiprod{^R{\rm D}^{\mu} \eta_u,\eta_u}\,dt&\le C k_n\max_{j=1}^{n} k_j^{2\ell-2-\mu}\Big(\int_{I_j}\Big\|\frac{\partial^\ell u}{\partial t^\ell} \Big\|\, dt\Big)^2\quad {\rm by~Lemma}~ \ref{lem: estimat of Dalpha eta},\\
\|\nabla \xi(t)\|&\le C\,h\|u(t)\|_2\quad  {\rm by~the~elliptic~projection~error}~ \eqref{eq:estimate ritz projector 1},\\
\|\nabla\eta_u(t)\|&\le C\, k_n^{\ell-1}\int_{I_n} \Big\|\nabla \frac{\partial^\ell  u}{\partial t^\ell} \Big\|\,ds \quad {\rm by~the~time~projection~error}~ \eqref{projection error},
 \end{align*}
where in the first inequality we also used the non-increasing time step assumption. So,
\begin{equation}\label{eq:orthogonality2}
\begin{aligned}
  \int_0^T|\zeta|_1^2dt
&\le  C\,h^4\|u_0\|_2^2+C\int_0^{T}
\bigiprod{^R{\rm D}^{\mu} \pw\xi,\pw\xi}dt+C\,h^2k^2\int_0^T\|u\|_2^2dt\\
& +C\max_{n=1}^N \,k_n^{2\ell-2-\mu}\Big(\Big(\int_{I_n}\Big\|\frac{\partial^\ell u}{\partial t^\ell} \Big\|\,dt\Big)^2+k_n^{\mu}\Big(\int_{I_n}\Big\|\nabla \frac{\partial^\ell u}{\partial t^\ell} \Big\|\,dt\Big)^2\Big)\,.
\end{aligned}
\end{equation}
It remains to estimate $\int_0^T
\bigiprod{^R{\rm D}^{\mu} \pw\xi,\pw\xi}\,dt$. From the decomposition:
\begin{equation}\label{eq: decomposition pw xi}
\int_{I_n}
\bigiprod{^R{\rm D}^{\mu} \pw\xi,\pw\xi}\,dt=
\int_{I_n}
\left[\bigiprod{^R{\rm D}^{\mu} \eta_\xi,\eta_\xi}+
\bigiprod{^R{\rm D}^{\mu} \xi,\eta_\xi}+
\bigiprod{^R{\rm D}^{\mu} \pw\xi,\xi}\right]\,dt\,.\end{equation}
By  Lemma \ref{lem: estimat of Dalpha eta},
\[
\int_{I_n} \bigiprod{{\rm ^R D}^\mu  \eta_\xi,\eta_\xi}\,dt \leq
C\,k_n^{1-\mu}\max_{j=1}^{n}
 \Big(\int_{I_j}\| \xi'\|\, dt\Big)^2\le C\,k_n\max_{j=1}^{n}
 \Big(k_j^{-\frac{\mu}{2}}\int_{I_j}\| \xi'\|\, dt\Big)^2\,.\]
Inserting the above bound in \eqref{eq: decomposition pw xi}, then summing over $n$ and  using the achieved bound in  \eqref{eq: bound 2}, we obtain
\begin{equation}\label{eq: estimate of pi xi}
\int_0^T
|\bigiprod{^R{\rm D}^{\mu} \pw\xi,\pw\xi}|\,dt\le
C\max_{n=1}^{N}
 \Big(k_n^{-\frac{\mu}{2}}\int_{I_n}\| \xi'\|\, dt\Big)^2+C\Big(\|\xi(0)\|+\int_0^T\|\xi'\|\,dt\Big)^2.\end{equation}
Finally, to complete the proof, we combine \eqref{eq:orthogonality2} and \eqref{eq: estimate of pi xi}. $\quad \Box$
\end{proof}

In the next theorem we show our main convergence results of the   DG FE solution.
Typically,
the exact solution $u$ of problem \eqref{eq: ivp} satisfies the finite regularity assumptions:
\begin{equation} \label{eq:countable-regularity v1}
\| u'(t)\|_2+t\| u''(t)\|_1\le {\bf M}\,\, t^{\sigma-1}\quad {\rm for}~~~t>0,
\end{equation}
for some positive constants ${\bf M}$  and $\sigma$. Due to the singular behaviour $u$ near $t=0$,  we employ a family of non-uniform meshes, where the
time-steps are graded towards $t=0$; see~\cite{McLeanMustapha2007,MustaphaMcLean2011}. More
precisely, for a fixed parameter $\gamma\ge1$,  we assume that
\begin{equation}\label{eq: tn standard}
t_n=(n/N)^\gamma T\quad\text{for $0\le n \le N$.}
\end{equation}
One can easily see that the sequence of  time-step sizes $\{k_j\}_{j=1}^N$ is nondecreasing, that is, $k_i\le k_j$ for $1\le i\le j\le N$. One can also show the following mesh property:
\begin{equation}\label{eq: time mesh property}
k_j\le \gamma k t_j^{1-1/\gamma}\,.\end{equation}

\begin{theorem}\label{thm:error-bound-abstract}
Let $u$ be the solution of (\ref{eq: ivp}) satisfying the regularity property \eqref{eq:countable-regularity v1} with $\sigma>\mu/2$. Let $U$ be the DG FE solution defined by~\eqref{eq: DGFM}. Then, we have
\[
\int_0^T \|U-u\|^2 \,dt \le C\,(h^4+
    k^{\gamma(2\sigma-\mu)})\quad{\rm for}\quad
  1\le \gamma \le   \frac{4-\mu}{2\sigma-\mu}\,
\] where $C$ is a constant that depends on $T$, $\mu$,  $\gamma$, $\sigma$, and on ${\bf M}$.
\end{theorem}
\begin{proof}
From the decomposition of the error  in (\ref{eq: U-u=theta+eta}), the triangle inequality, the  bound in Theorem \ref{thm: estimate of theta}, the inequality $\|\pw \xi\|_{L^2(L^2)}\le \|\xi\|_{L^2(L^2)}$ by \eqref{(ii)}, the elliptic projection error
\eqref{eq:estimate ritz projector 1},  the error from the time projection  \eqref{projection error}, we have
\[
\int_0^T\|U-u\|^2\,dt \le C\Big(h^4C_1(k,u)+
 C_2(k,u)
+h^2(h^2+k^2)\|u\|_{L^2(H^2)}^2\Big)\,.
\]
By the definitions of $C_1(k,u)$ and $C_2(k,u)$ in \eqref{eq: C1C2}, the  regularity assumption \eqref{eq:countable-regularity v1}, and the inequality $h^2k^2\le \frac{1}{2}(h^4+k^4)$,  we observe
\[
\begin{aligned}
\int_0^T\|U-u\|^2\,dt
&\le Ch^4\max_{n=1}^{N}
 \Big(k_n^{-\frac{\mu}{2}}\int_{I_n}t^{\sigma-1}\, dt\Big)^2+Ch^4\Big(1+\int_0^Tt^{\sigma-1}\,dt\Big)^2\\
& +C\,k_1^{-\mu}\left(\int_{I_1}t^{\sigma-1}\,dt\right)^2+C\max_{n=2}^N \,k_n^{2-\mu}\left(\int_{I_n}t^{\sigma-2}\,dt\right)^2+C\,h^2k^2\\
&\le C(h^4\max_{n=1}^{N}
 k_n^{2\sigma-\mu}+h^4 +k_1^{2\sigma-\mu}+\max_{n=2}^N \,k_n^{4-\mu}t_n^{2\sigma-4}+k^4)\\
&\le C\,(h^4+k^{\min\{\gamma(2\sigma-\mu),4-\mu\} })
\end{aligned}\]
where in the last inequality,  by the mesh property \eqref{eq: time mesh property},  we used
\begin{align*}
k_n^{4-\mu}{t_n}^{2(\sigma-2)}&\le C\,k^{4-\mu} t_n^{2(\sigma-2)+4 -\mu-(4-\mu)/\gamma}
       \le C\,k^{\min\{\gamma(2\sigma-\mu),4-\mu\}}.\quad \Box\end{align*}
\end{proof}

\section{Numerical results}\label{sec: Numerical}
We present a sample of numerical tests using a model problem in one space dimension, of the form~\eqref{eq: ivp} with $\Omega=(0,1)$,  $[0,T]=[0,1],$ and $\A(x,t)=1+t^{3/2}$. We choose $u_0(x)=\sin(\pi x)$ for the initial data and choose the source term $f$
so that
\begin{equation}\label{eq:num ex3}
u(t)=  (1+t^{1-\mu}) \sin(\pi x)\,.
\end{equation}
 One easily verifies that the regularity
condition \eqref{eq:countable-regularity v1}  holds for $\sigma=1-\mu$.

The numerical  tests below  reveal faster rates of convergence than those suggested by Theorem \ref{thm:error-bound-abstract}, and that our regularity assumptions are more restrictive than is needed in practice. More precisely, Theorem \ref{thm:error-bound-abstract} shows  suboptimal (in time) convergence of order $O(k^{2-\frac{\mu}{2}}+h^2)$ for sufficiently graded time meshes in the time-space $L^2$-norm.  However, we demonstrate numerically optimal (in both time and space) rates of convergence in the stronger $L^\infty(L^2)$-norm. To this end,  We introduce a finer mesh
\begin{equation}\label{eq: fine grid}
\G^{m}=\{\,t_{j-1}+\ell k_j/m:\text{$j=1$, 2, \dots, $N$ and
    $\ell=0$, 1, \dots, $m$}\,\},
\end{equation}
and define the discrete maximum norm~$\|v\|_{\G^m}=\max_{t\in\G^{m}}\|v(t)\|$, so that, for sufficiently large values of~$m$, $\|U_h-u\|_{\G^m}$ approximates the uniform error~$\|U_h-u\|_{L^\infty(L^2)}$. In all tables, we choose $m=10.$

For the numerical illustration of the convergence rates in time, we choose  $M$ (the number of uniform spatial subintervals) to be sufficiently large such that  the spatial error is negligible compared to the error from the time discretization. We employ a time mesh of the form~\eqref{eq: tn standard}. Tables \ref{tab: mu=0.3}, \ref{tab: mu=0.5} and \ref{tab: mu=2/3 and 0.7} show the error (in the stronger $L^\infty(L^2)$) and the rates of convergence  when $\mu=0.3$, $0.5$, $2/3$ ~and $0.7$ respectively, for various choices of $N$~and $\gamma$.  We observe optimal rates of order $O(k^{\gamma\sigma})$ for various choices of $1\le \gamma \le   \frac{2}{\sigma}$  which is faster than the rate $O(k^{\frac{\gamma}{2}(2\sigma-\mu)})$ for $1\le \gamma \le   \frac{4-\mu}{2\sigma-\mu}$  predicted by our theory in Theorem \ref{thm:error-bound-abstract}. Noting that, in   Table  \ref{tab: mu=2/3 and 0.7},   $\sigma\le \mu$ and thus the assumption $\sigma >\mu/2$ in this theorem is not sharp.
\begin{table}
\renewcommand{\arraystretch}{1}
\begin{center}
\caption{Errors and time convergence rates with $\mu=0.3$ for various choices of  $\gamma$.} \label{tab: mu=0.3}
\begin{tabular}{|r|rr|rr|rr|}
\hline {$N$}&\multicolumn{2}{c|}{$\gamma=1$} &\multicolumn{2}{c|}{$\gamma=2$}
&\multicolumn{2}{c|}{$\gamma=3$}\\
\hline
   10& 5.8997e-03&        & 1.1252e-03&       & 9.9332e-04&\\
   20& 3.5981e-03& 0.71339& 4.1163e-04& 1.4507& 2.5524e-04& 1.9604\\
   40& 2.1827e-03& 0.72111& 1.5008e-04& 1.4556& 6.4530e-05& 1.9838\\
   80& 1.3208e-03& 0.72468& 5.4700e-05& 1.4562& 1.6137e-05& 1.9996\\
  160& 7.9804e-04& 0.72692& 1.9995e-05& 1.4519& 4.0085e-06& 2.0092\\
  320& 4.8168e-04& 0.72840& 7.3478e-06& 1.4443& 9.9164e-07& 2.0152\\
\hline
\end{tabular}
\end{center}
\end{table}

\begin{table}
\renewcommand{\arraystretch}{1}
\begin{center}
\caption{Errors and time convergence rates with $\mu=0.5$ for various choices of $\gamma$.}  \label{tab: mu=0.5}
\begin{tabular}{|r|rr|rr|rr|rr|}
\hline {$N$}&\multicolumn{2}{c|}{$\gamma=1$} &\multicolumn{2}{c|}{$\gamma=2$}
&\multicolumn{2}{c|}{$\gamma=3$}&\multicolumn{2}{c|}{$\gamma=4$}\\
\hline
   10& 1.149e-02&      & 3.262e-03&      & 1.560e-03&      & 1.882e-03&     \\
   20& 7.641e-03& 0.589& 1.619e-03& 1.011& 5.972e-04& 1.385& 4.869e-04&   1.951\\
   40& 5.151e-03& 0.569& 8.037e-04& 1.010& 2.192e-04& 1.446& 1.209e-04&   2.009\\
   80& 3.641e-03& 0.500& 3.997e-04& 1.008& 7.867e-05& 1.478& 2.933e-05&   2.044\\
  160& 2.570e-03& 0.503& 1.992e-04& 1.005& 2.797e-05& 1.492& 7.011e-06&   2.064\\
  320& 1.812e-03& 0.504& 9.940e-05& 1.003& 9.908e-06& 1.497& 1.774e-06&   1.982\\
  \hline
\end{tabular}
\end{center}
\end{table}

\begin{table}
\renewcommand{\arraystretch}{1}
\begin{center}
\caption{Errors and time convergence rates   for various choices of $\gamma$.}  \label{tab: mu=2/3 and 0.7}
\begin{tabular}{|r|rr|rr|rr|rr|}
\hline
\multicolumn{9}{c}{$\mu=2/3$}\\
\hline {$N$}&\multicolumn{2}{c|}{$\gamma=1$} &\multicolumn{2}{c|}{$\gamma=2$}
&\multicolumn{2}{c|}{$\gamma=4$}&\multicolumn{2}{c|}{$\gamma=6$}\\
\hline
 10& 1.677e-02&      & 7.579e-03&      & 3.416e-03&      & 3.261e-03&      \\
 20& 1.327e-02& 0.338& 4.677e-03& 0.696& 1.393e-03& 1.294& 9.087e-04& 1.843\\
 40& 1.044e-02& 0.346& 3.036e-03& 0.623& 5.553e-04& 1.327& 2.471e-04& 1.879\\
 80& 8.191e-03& 0.350& 1.940e-03& 0.646& 2.205e-04& 1.332& 6.435e-05& 1.941\\
160& 6.427e-03& 0.350& 1.229e-03& 0.658& 8.753e-05& 1.333& 1.643e-05& 1.970\\
  \hline \hline
\multicolumn{9}{c}{$\mu=0.7$}\\
\hline {$N$}&\multicolumn{2}{c|}{$\gamma=1$} &\multicolumn{2}{c|}{$\gamma=3$}
&\multicolumn{2}{c|}{$\gamma=5$}&\multicolumn{2}{c|}{$\gamma=7$}\\
\hline
   10& 1.792e-02&      & 5.149e-03&          & 3.625e-03&      & 3.991e-03&       \\
   20& 1.446e-02& 0.309& 2.905e-03& 8.258e-01& 1.318e-03& 1.459& 1.121e-03& 1.832\\
   40& 1.160e-02& 0.318& 1.577e-03& 8.810e-01& 4.673e-04& 1.496& 3.052e-04& 1.877\\
   80& 9.290e-03& 0.321& 8.479e-04& 8.955e-01& 1.652e-04& 1.499& 7.981e-05& 1.935\\
  160& 7.447e-03& 0.319& 4.547e-04& 8.989e-01& 5.843e-05& 1.500& &\\ 
\hline
\end{tabular}
\end{center}
\end{table}
Next, we test the performance of the spatial FEs discretizaton of the scheme \eqref{eq: DGFM}. A uniform spatial
mesh that consists of $M$ subintervals where each is of width  $h$ will be used. We refine the  time mesh such that
the spatial error is dominating. By Theorem \ref{thm:error-bound-abstract}, a convergence of order $O(h^2)$ is expected. We illustrate
these results in Table \ref{tab: space convergence}.
\begin{table}
\renewcommand{\arraystretch}{1}
\begin{center}
\caption{Errors and  convergence rates in space with $\mu=0.3, 0.5$ and $0.7$.}  \label{tab: space convergence}
\begin{tabular}{|r|rr|rr|rr|}
\hline {$M$}&\multicolumn{2}{c|}{$\mu=0.3$} &\multicolumn{2}{c|}{$\mu=0.5$}
&\multicolumn{2}{c|}{$\mu=0.7$}\\
\hline
 10& 1.2156e-02&       & 1.2780e-02&       & 1.2563e-02&        \\
 20& 3.1130e-03& 1.9653& 3.2743e-03& 1.9646& 3.1768e-03& 1.9836\\
 40& 7.8803e-04& 1.9820& 8.2897e-04& 1.9818& 7.9873e-04& 1.9918\\
 80& 1.9826e-04& 1.9909& 2.0864e-04& 1.9903& 2.0029e-04& 1.9956\\
160& 4.9724e-05& 1.9954& 5.2355e-05& 1.9946& 5.1065e-05& 1.9717\\
\hline
\end{tabular}
\end{center}
\end{table}

\providecommand{\bysame}{\leavevmode\hbox to3em{\hrulefill}\thinspace} \providecommand{\MR}{\relax\ifhmode\unskip\space\fi MR }
\providecommand{\MRhref}[2]{%
  \href{http://www.ams.org/mathscinet-getitem?mr=#1}{#2}
} \providecommand{\href}[2]{#2}


\end{document}